\documentclass[11pt]{article}  

\usepackage{amsmath}   
\usepackage{amsthm}   
\usepackage{graphicx}
\usepackage{color}
\usepackage{url}		

\theoremstyle{plain}                            
\numberwithin{equation}{section}
\newtheorem{thm}{Theorem}[section]
\newtheorem{theorem}[thm]{Theorem}
\newtheorem{lemma}[thm]{Lemma}

\graphicspath{%
    {/}% inserted by PCTeX
}

\usepackage[implicit=false,colorlinks=true,urlcolor=red]{hyperref}  
\begin{document}

\title{An Exploration of Sequence A000975}
\author{Paul K. Stockmeyer \\ \ \\
Department of Computer Science\\
             College of William and Mary\\
               Williamsburg, Virginia 23187\\
               USA\\
Email: \href{mailto:stockmeyer@cs.wm.edu}{stockmeyer@cs.wm.edu}}
\date{}
\maketitle

\begin{abstract}
Sequence A000975 in the Online Encyclopedia of Integer Sequences (OEIS)
starts out 1, 2, 5, 10, 21, 42, 85, \dots\ .  As of July 1, 2016, the description %of this sequence 
in the OEIS lists several characterizations of this sequence and numerous examples of
instances where this sequence occurs.  It also presents a ``not yet proved'' result, a 
conjecture, and an unanswered question concerning this sequence.  In this paper we show 
that all of these proposed results are in fact true.
\end{abstract}

\section{Characterizations of Sequence A000975}

Sequence \href{https://oeis.org/A000975}{A000975} in the Online Encyclopedia of Integer Sequences (OEIS)  \cite{OEIS} starts out 1, 2, 5, 10, 21, 42, 85, \dots\ .
Throughout this paper we denote this sequence as sequence $A$ and the $n$th term of this sequence as $A(n)$, written in function notation. 
The OEIS defines this sequence recursively as follows.

\medskip\noindent{\bf Characterization 1:}

\medskip 1.  $A(1) = 1$,

2.  $A(2n) = 2A(2n\!-\!1)$, and

3.  $A(2n\!+\!1) = 2A(2n)+1$.

%\begin{enumerate}
%\item $A(1) = 1$;
%\item $A(2n) = 2A(2n\!-\!1)$, and
%\item $A(2n\!+\!1) = 2A(2n)+1$.
%\end{enumerate}

\medskip\noindent
Thus to get the next term we either double the preceding value or double it and add 1.  But doubling is the same as appending a 0 in base 2, while doubling and adding 1 is the same as appending a 1 in base 2.  This provides a second characterization.

\medskip\noindent{\bf Characterization 2:}  For $n>0$, $A(n)$ is the number whose binary representation has length $n$ and consists of
alternating ones and zeros.

\medskip\noindent We illustrate this characterization in Table 1.
\begin{table}[h!]
$$
\begin{array}{l@{\hspace{5pt}}c@{\hspace{5pt}}r@{\hspace{5pt}}c@{\hspace{5pt}}r}
A(1) &=& 1&=&1_{(2)}\\
A(2) &=& 2&=&10_{(2)}\\
A(3) &=& 5&=&101_{(2)}\\
A(4) &=& 10&=&1010_{(2)}\\
A(5) &=& 21&=&10101_{(2)}\\
A(6) &=& 42&=&101010_{(2)}\\
A(7) &= &85&=&1010101_{(2)}\\
A(8) &=&170&=&10101010_{(2)}\\
A(9) &=& 341&=&101010101_{(2)}\\
A(10) &=& 682&=&1010101010_{(2)}
\end{array}
$$
\caption{The sequence $A$ written in binary.}
\end{table}

If we add $A(n)$ to $A(n-1)$, the binary representation of the sum is a string of $n$ ones, with value $2^n\!-\!1$.  This
gives us our third characterization.

\newpage
\medskip\noindent{\bf Characterization 3:} 

\medskip 
1.  $A(1) = 1$, and

2.  $A(n) = \left(2^n\! -\!1\right)- A(n\!-\!1)$ for $n> 1$.

\medskip

%\begin{enumerate}
%\item $A(1) = 1$; and
%\item $A(n) = \left(2^n\! -\!1\right)- A(n\!-\!1)$ for $n> 1$.
%\end{enumerate}

Alternatively, if we subtract $A(n\!-\!2)$ from $A(n)$, most of the ones in the binary representations cancel, leaving $2^{n-1}$.
This gives us yet another characterization.

\medskip\noindent{\bf Characterization 4:}

\medskip
1.  $A(1)=1$,

2.  $A(2)=2$, and

3.  $A(n)= A(n\!-\!2) + 2^{n-1}$ for $n>2$.

\medskip
%\begin{enumerate}
%\item $A(1)=1$;
%\item $A(2)=3$; and
%\item $A(n)= A(n\!-\!2) + 2^{n-1}$ for $n>2$.
%\end{enumerate}

There are several standard methods for obtaining a closed form expression for $A(n)$.  We present a rather unusual
derivation.  Recall that the binary representation of the fraction $\frac23$ is
$$
\frac23 = 0.10101010\dots\ _{(2)}.
$$
Multiplying this value by $2^n$ moves the binary point $n$ places to the right.  Rounding down then truncates this expression,
yielding $A(n)$, according to Characterization 2, and giving us our next characterization.

\medskip\noindent{\bf Characterization 5:}  $A(n) = \left\lfloor{\displaystyle \frac{2}{3}}(2^n)\right\rfloor$ for all $n\ge 1$.

\medskip\noindent Other expressions are
\begin{eqnarray*}
A(n) 
     &=& \left\{
              \begin{array}{ll}
                \frac{\textstyle 2^{n+1} -1 \rule[-3pt]{0pt}{10pt}}{\textstyle 3 \rule{0pt}{9pt}}&\mbox{if $n$ odd}\\ \ \\            
                \frac{\textstyle 2^{n+1} -2 \rule[-3pt]{0pt}{10pt}}{\textstyle 3 \rule{0pt}{9pt}}&\mbox{if $n$ even}
              \end{array} 
        \right.\\ \ \\ \ \\
     &=&\frac{2^{n+2} - 3 - (-1)^n}{6}.
\end{eqnarray*}
It is an easy exercise to prove that all these characterizations of $A(n)$ are equivalent.  We observe 
that $A(n)$ is even or odd exactly when $n$ is even or odd, respectively.

\section{Occurrences of Sequence A000975}

Why do we care about a particular sequence?  Often it is because the sequence occurs as the answer to some counting problem.  The OEIS lists several places where the values in sequence \href{https://oeis.org/A000975}{A000975} occur.  
Here are some of the more interesting.

\medskip\noindent{\bf Occurrence 1:}  
$A(n)$ is the number of moves needed to solve the $n$-ring Chinese Rings 
puzzle (baguenaudier) if the rings are moved one at a time. See, for example, \cite[Chapter 1]{Hinz}.

\medskip
Figure \ref{rings} shows the state graph for the 4-ring puzzle.  The 16 vertices represent the possible states of the puzzle, and two vertices are joined by an edge if the two corresponding states are one move apart.  The labels on the vertices denote the
positions of the 4 rings in that state, with 0 representing a ring that is off the sliding bar and 1 a ring that is on.

\begin{figure}[h]
\centerline{\includegraphics[width=5.02in,height=1.32in,keepaspectratio]{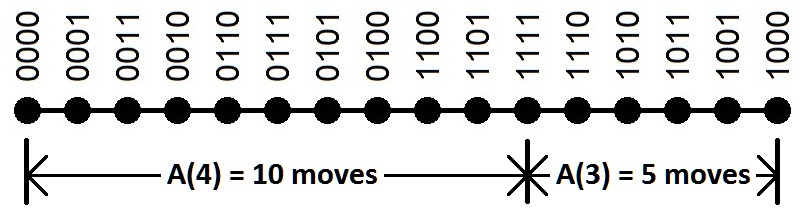}}
\caption{The state graph for the 4-ring  Chinese Rings puzzle.}
\label{rings}
\end{figure}

It is easy to confirm that the state graph for the $n$-ring puzzle is a path of length $2^n-1$ from the state labeled $0^n$ 
to the state labeled $10^{n-1}$.  The sub-path from state $0^n$ to state $1^n$ constitutes an optimal solution of the
$n$-ring puzzle, while the
sub-path from state $1^n$ to state $10^{n-1}$ represents an optimal solution of the $(n\!-\!1)$-ring puzzle. The number of moves
in an optimal solution to this puzzle thus satisfies Characterization 3 of sequence $A$.
%The number $M(n)$ of moves made by the usual algorithm for solving this puzzle satisfies the recursive equation 
%$M(n) = M(n\!-\!1) + 2M(n\!-\!2) + 1$,
%with $M(0)=0$ and $M(1)=1$.  It is an easy exercise to show that this is yet another characterization of our sequence $A$.

\medskip
Figure \ref{rings} also illustrates a related occurrence of our sequence.

\medskip\noindent{\bf Occurrence 2:}  $A(n)$ is the distance between a string of $n$ zeros and a string of $n$ ones in the standard $n$-bit binary Gray code. Again, see, for example, \cite[Chapter 1]{Hinz}.

%\medskip\noindent
%Note that if we ignore the leading 1 on the labels on the right-hand third of the graph, we have an example of both 
%of these occurrences for the case $n=3$, confirming Characterization 3 of Section 1.
\bigskip
A special case of a problem posed by Donald Knuth \cite{knuth} and solved by O.~P.~Lossers \cite{Lossers} provides our next occurrence

\medskip\noindent{\bf Occurrence 3:}  $A(n)$ is the number of ways to partition a set of $n\!+\!2$ people sitting around a circular table
into three affinity groups with no two members of a group seated next to each other.

\medskip
We illustrate this occurrence in Figure \ref{dk}, showing the $A(4)=10$ partitions of six people into three affinity groups.  For
concreteness we assign the names A and B to the affinity groups of the people at the bottom left and bottom right, respectively.

\begin{figure}[h]
\centerline{\includegraphics[width=5.1in,height=1.94in,keepaspectratio]{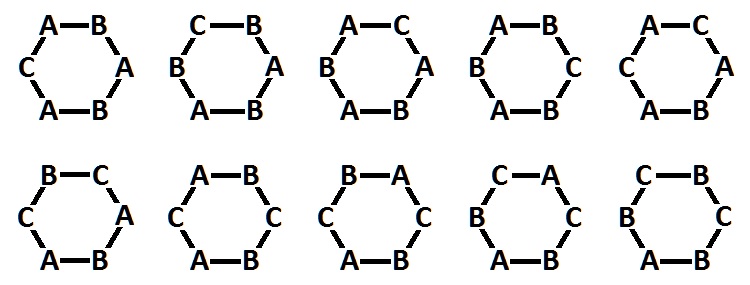}}
\caption{The ten partitions of 6 people into three affinity groups.}\label{dk}
\end{figure}

Following \cite{Lossers}, we note that there are $2^n\!-\!1$ strings of length $n\!+\!2$ that start with AB, contain all the
letters A, B, and C, and have no two adjacent letters the same.  If such a string ends in either B or C, it can be wrapped 
into a circle to represent an affinity partition of $n\!+\!2$ people.  If such a string ends in A, the A can be deleted and the shortened
string will 
represent an affinity partition of $n\!+\!1$ people.  The number of affinity partitions of $n\!+\!2$ people is thus another problem that satisfies Characterization 3 of sequence $A$.

\medskip
Equivalently, $A(n)$ is the number of different 3-colorings for the vertices of all triangulated $(n\!+\!2)$-gons if the colors of the two base vertices are fixed.

\section{A ``not yet proved'' property proposed by Antti Karttunen}

Antti Karttunen has suggested that our sequence \href{https://oeis.org/A000975}{A000975} serves as a link between 
sequences \href{https://oeis.org/A0000217}{A0000217} (the triangular numbers)
and \href{https://oeis.org/A048702}{A048702} (the even length binary palindromes divided by 3) in the OEIS.  The triangular numbers, which we represent 
by $T(n)$, are
well-known:  $T(n) = n(n\!+\!1)/2$.  % for $n\ge 0$.  
The even length binary palindromes, denoted $P(n)$, (sequence \href{https://oeis.org/A048701}{A048701} in the OEIS), are less well known.  
We illustrate the sequence $P(n)/3$ 
in Table \ref{pal}.

\begin{table}[h]
$$
\begin{array}{l@{\hspace{5pt}}c@{\hspace{5pt}}r@{\hspace{5pt}}r@{\hspace{5pt}}r}
%P(0)/3&=& &0/3 =&0\\
P(1)/3&=&11_{(2)}/3 = &3/3=&1\\
P(2)/3&=&1001_{(2)}/3 =& 9/3=&3\\
P(3)/3&=&1111_{(2)}/3 =&15/3=&5\\
P(4)/3&=&100001_{(2)}/3=&33/3=&11\\
P(5)/3&=&101101_{(2)}/3=&45/3=&15\\
P(6)/3&=&110011_{(2)}/3=&51/3=&17\\
P(7)/3&=&111111_{(2)}/3=&63/3=&21\\
P(8)/3&=&10000001_{(2)}/3=&129/3=&43\\
P(9)/3&=&10011001_{(2)}/3=&153/3=&51\\
P(10)/3&=&10100101_{(2)}/3=&165/3=&55\\
\end{array}
$$
\caption{The sequence of even length binary palindromes divided by 3.}
\label{pal}
\end{table}

The proposal of  Karttunen is that the values $A(k)$ serve as the indices $n$ where $T(n)$ and $P(n)/3$ agree, as illustrated in 
Table \ref{blue}.

\begin{table}[h!]
$$
\begin{array}{r|rrrrrrrrrrrrrrr}
n&{\bf \color{blue}1}&{\bf \color{blue}2}&3&4&{\bf \color{blue}5}&6&7&8&9&{\bf \color{blue}10}&11&12&13&14&15
           \rule[-5pt]{0pt}{15pt}\\ \hline
T(n)&{\bf \color{blue}1}&{\bf \color{blue} 3}&6&10&{\bf \color{blue} 15}&21&28&36&45&{\bf \color{blue} 55}&66&78&91&105&120
             \rule[-5pt]{0pt}{15pt}\\ \hline
P(n)/3&{\bf \color{blue} 1}&{\bf \color{blue} 3}&5&11&{\bf \color{blue} 15}&17&21&43&51&{\bf \color{blue} 55}&63&65&73&77&85
            \rule[-5pt]{0pt}{15pt}
\end{array}
$$
\caption{Common values of $T(n)$ and $P(n)/3$.}
\label{blue}
\end{table}

%$$
%\begin{array}{r|r@{\hspace{12pt}}|r@{\hspace{20pt}}}
%n&\multicolumn{1}{c|}{T(n)}&\multicolumn{1}{c}{P(n)/3}\rule[-5pt]{0pt}{12pt}\\ \hline
%%{\bf 0}&{\bf 0}&{\bf 0}\rule{0pt}{15pt}\\
%{\bf 1}&{\bf 1}&{\bf 1}\rule{0pt}{11pt}\\
%{\bf 2}&{\bf 3}&{\bf 3}\\
%3&6&5\\
%4&10&11\\
%{\bf 5}&{\bf 15}&{\bf 15}\\
%6&21&17\\
%7&28&21\\
%8&36&43\\
%9&45&51\\
%{\bf 10}&{\bf 55}&{\bf 55}
%\end{array}
%$$

\begin{lemma} For any positive integer $n$ let $k=\lfloor\log_2(n)\rfloor+1$. Then the binary representation of $n$ is $k$
bits long, and $P(n)$ satisfies
$$P(n) = n2^k + R(n),$$ 
where $R(n)$ is the binary reversal of $n$ {\rm(}sequence \href{https://oeis.org/A030101}{$A030101$} in the OEIS{\rm)}.
\end{lemma}

\noindent For example, when $n=12=1100_{(2)}$ we have $k=4$  and $R(12) = 0011_{(2)} = 3$, so
$P(12) = 11000011_{(2)}=12\left( 2^4\right) + 3 = 195.$  This Lemma follows immediately from the definition of $P(n)$.

\begin{theorem}  For all $n>0$, we have $T(n) = P(n)/3$ if and only if $n = A(k)$ for some positive integer $k$.
\end{theorem}

\begin{proof}
First suppose $n = A(k)$ where $k$ is odd.  
Then we know that $$n = \left(2^{k+1}\! -\!1\right)/3,$$ $n$ is $k$ bits long, and $R(n) = n$.   We have
\begin{eqnarray*}
P(n)/3 &=&\left(n2^{k}\! + n\right)/3\\
&=&n\left(2^k\! + 1\right)/3\\
&=&\left(\frac{n}{2}\right)\left(2^{k+1}\! + 2\right)/3\\
&=&\frac{n(n\!+\!1)}{2} = T(n).
\end{eqnarray*}

Now suppose $n = A(k)$ where $k$ is even.  
Then we know that $$n = \left(2^{k+1}\! -\!2\right)/3,$$ $n$ is $k$ bits long, and $R(n) = n/2$.  We have
\begin{eqnarray*}
P(n)/3 &=&\left(n2^{k}\! + n/2\right)/3\\
%&=&n\left(2^k\! +\! 1\right)/3\\
&=&\left(\frac{n}{2}\right)\left(2^{k+1}\! + 1\right)/3\\
&=&\frac{n(n\!+\!1)}{2} = T(n).
\end{eqnarray*}

Next suppose $2^{k-1}\le n \le A(k)-1$.  Then $n$ is $k$ bits long, and
\begin{align*}
P(n)/3 &= \left(n2^{k} + R(n)\right)/3\\
       &>\left(n2^k\right)/3\\
       &=\left(\frac{n}{2}\right)\left(\frac{2^{k+1}}{3}\right)\\
       &>\left(\frac{n}{2}\right)A(k)\\
       &\ge \frac{n(n+1)}{2} = T(n).
\end{align*}

Finally, suppose $A(k)+1 \le n \le 2^k-1$.  Then $n$ is $k$ bits long, and
\begin{align*}
P(n)/3 &= \left(n2^{k} + R(n)\right)/3\\
       &<\left(n2^k + 2^k \right)/3\\
       &=\left(\frac{n+1}{2}\right)\left(\frac{2^{k+1}}{3}\right)\\
       &<\left(\frac{n+1}{2}\right)\left(A(k)+1\rule{0pt}{10pt}\right)\\
       &\le \frac{(n+1)n}{2} = T(n). 
\end{align*}
\end{proof}

\section{The conjectures of Reinhard Zumkeller}

Reinhard Zumkeller has conjectured that our sequence $A$ is related to the sequence \href{https://oeis.org/A265158}{A265158}
 in the OEIS.  We denote this sequence as $B$ and its $n$th term as $B(n)$.  The sequence is defined by

\medskip
1.  $B(1)= 1$,

2.  $B(2n)= 2B(n)$ for $n\ge 1$, and

3.  $B(2n+1) = \lfloor B(n)/2)\rfloor$ for $n\ge 1$.
%\begin{align*}
%B(1) &=1\mbox{,}\\ 
%B(2n) &= 2B(n)\mbox{ for $n\ge 1$, and }\\ 
%B(2n+1) &= \lfloor B(n)/2)\rfloor \mbox{ for $n\ge 1$}.
%\end{align*}

\medskip\noindent
The first few values of this sequence are displayed in Table \ref{B}.
\begin{table}[h]
$$\begin{array}{rr|rr|rr|rr}
n&B(n)&\multicolumn{1}{c}n&B(n)&\multicolumn{1}{c}n&B(n)&\multicolumn{1}{c}n&B(n)\rule[-5pt]{0pt}{12pt}\\ \hline
1&1\;&9&2\;&17&4\;&25&0\;\rule{0pt}{11pt}\\
2&2\;&10&2\;&18&4\;&26&0\;\\
3&0\;&11&0\;&19&1\;&27&0\;\\
4&4\;&12&0\;&20&4\;&28&0\;\\
5&1\;&13&0\;&21&1\;&29&0\;\\
6&0\;&14&0\;&22&0\;&30&0\;\\
7&0\;&15&0\;&23&0\;&31&0\;\\
8&8\;&16&16\;&24&0\;&32&32\;
\end{array}$$
\caption{The sequence $B$.}
\label{B}
\end{table}

A notable property of this sequence is that it contains long runs of zeros.  Zumkeller conjectures that the lengths 
of the record runs 
of zeros in sequence $B$ are exactly the values found in sequence $A$.  The run lengths themselves form sequence 
\href{https://oeis.org/A264784}{A264784}, with values
{\bf \color{blue} 1}, {\bf \color{blue} 2}, {\bf \color{blue} 5}, {\bf \color{blue} 10}, 1, 
{\bf \color{blue} 21}, 2, 
{\bf \color{blue} 42}, 1, 1, 5, 1, 
{\bf \color{blue} 85}, 2, 2, 10, 2, {\bf \color{blue} 170}, 1, 1, 1, 5, 1, 1, 5, 1, 21, 1, 1, 5, 1, {\bf \color{blue}341}, 
2, 2, 2, 10, 2, 2, 10, 2, 42, 2, 2, 10, 2, {\bf \color{blue} 682}, 1, 1, 1, 1, 5, 1, 1, 1, 5, 1, 1, 5, 1, 21, 
1, 1, 1, 5, 1, 1, 5, 1, 21, 1, 1, 5, 1, 85, 1, 1, 1, 5, 1, 1, 5, 1, 21, 1, 1, 5, 1, {\bf \color{blue} 1365}, \dots\ .
We will refer to this sequence of run lengths as sequence $R$, with entries $R(n)$.

\newpage
We confirm a few facts about sequence $B$.
\begin{lemma}
$$
B(A(k)) = \left\{\begin{array}{ll}
           1& \mbox{ if $k$ is odd}\rule[-5pt]{0pt}{5pt}\\ 
           2& \mbox{ if $k$ is even}\rule{0pt}{15pt}           
          \end{array}\right.
$$
\end{lemma}

\begin{proof}
The proof is by induction on $k$.  For $k=1$ we have $B(A(1)) = B(1) = 1.$  Let $k$ be an even positive integer, and
suppose that the lemma is true for smaller values of $k$. We have
\begin{align*}
                 B(A(k)) &= B\left( (2^{k+1} - 2)/3\right) \\
                         &= 2B\left( (2^k -1)/3\right)\\
                         &=2B\left(A(k\!-\!1\rule{0pt}{10pt})\right)\\
                         &= 2(1) = 2.
\end{align*}
Now let $k\ge 3$ be an odd integer, and suppose that the lemma is true for smaller values of $k$.  We have
\begin{align*}
                 B(A(k)) &= B\left((2^{k+1} - 1)/3\right) \\
                         &= \left\lfloor B\left((2^k-2)/3\right)/2\right\rfloor\\
                         &= \left\lfloor B\left(A(k\!-\!1\rule{0pt}{10pt})\right)/2\right\rfloor\\
                         &= \left\lfloor 2/2 \right\rfloor = 1.
\end{align*}
\end{proof}

\begin{lemma}
$B(A(k) + 1) = 0$ for $k\ge 2$.
\end{lemma}

\begin{proof}
For $k$ even we have 
\begin{align*}
   B(A(k)+1) &= B\left((2^{k+1} -2)/3+ 1\right)\\
             &= B\left((2^{k+1} + 1)/3\right)\\
             &= \left\lfloor B\left((2^k - 1)/3\right)/2\right\rfloor\\
             &= \left\lfloor B\left(A(k\!-\!1)\rule{0pt}{10pt}\right)/2\right\rfloor\\
             &= \lfloor 1/2\rfloor = 0.
\end{align*}
Once the case $k$ even is established, the case $k$ odd follows.
\begin{align*}
  B(A(k)+1) &= B\left( (2^{k+1}-1)/3 + 1\right)\\
            &= B\left( (2^{k+1} + 2)/3\right)\\
            &= 2B\left((2^k + 1)/3\right)\\
            &= 2B\left(A(k\!-\!1) + 1\rule{0pt}{10pt}\right)\\
            &= 2(0) = 0.
\end{align*}
\end{proof}

Note that Lemmas 4.1 and 4.2 together imply that for all $k\ge 2$ there is a run of zeros in sequence $B$ 
beginning at index $n=A(k)+1$.

\begin{lemma}
$B(2^k) = 2^k$ for all $k\ge 0$.
\end{lemma}
\begin{proof}
The proof is by induction on $k$.  For $k=0$ we have $B(2^0)=B(1) = 1 = 2^0$.  Now assuming the lemma is true for all values
smaller than some $k>0$, we have 
$$B(2^k) = 2B(2^{k-1}) = 2(2^{k-1}) = 2^k.$$
\end{proof}
\begin{theorem}
For every $k\ge 2$ we have $B(n) = 0$ for $A(k) + 1\le n\le 2^{k}-1$.
These values form a run of zeros of length $A(k\!-\!1)$ in $B(n)$, and these are the runs of record length.
\end{theorem}
\begin{proof}
We again use induction on $k$.  For $k=2$ we consider $n$ such that $A(2) + 1 = 3 \le n \le 3 = 2^2-1$, or $n=3$.  Now
$B(3)=0$ so the claim is true in this case.  We now consider the claim for some $k>2,$ and assume the claim is true for
smaller values of $k$.
We know from Lemma 4.2 that $B(n)=0$ for $n= A(k)+1$, so we need consider only those $n$ such that $A(k)+2\le n\le 2^k-1$.
If $n$ is even, this implies $A(k\!-\!1) + 1  \le n/2 \le 2^{k-1}-1$, so by the induction hypothesis we have
\begin{align*}
B(n) &= 2B(n/2)\\
     &= 2(0) = 0.
\end{align*}
If $n$ is odd, it can still be shown that $A(k\!-\!1) + 1 \le (n\!-\!1)/2 \le 2^{k-1}-1$, so the induction hypothesis yields
\begin{align*}
B(n) &= \left\lfloor B\left((n\!-\!1)/2\rule{0pt}{10pt}\right)/2\right\rfloor\\
     &= \lfloor 0/2\rfloor = 0.
\end{align*}

For all $k\ge 2$, then, we have a run of zeros in sequence $B(n)$ starting at $A(k) + 1$ and ending at $2^k - 1$,
of length $\left(2^k\!-\!1\right) - A(k) = A(k\!-\!1)$.  That these are precisely the runs of record length follows from the fact that $B(2^k)>0$ for all $k\ge 0$ and that each of these runs covers two-thirds of the entries between consecutive powers of 2.
\end{proof}

\medskip
We can say more about the number and lengths of runs of zeros in $B$ if we write the index $n$ in binary.  Using $\omega$ to
denote an arbitrary word, or string, over the alphabet $\mathcal{A} = \{0,1\}$, we can restate the definition of $B(n)$ as follows:  

\medskip
1.  $B(1)=1$,

2.  $B\left(\omega0_{(2)}\right) = 2B\left(\omega_{(2)}\right)$\rule[-8pt]{0pt}{24pt}, and

3.  $B\left(\omega1_{(2)}\right) = \left\lfloor B\left(\omega_{(2)}\right)/2)\right\rfloor$.

\medskip
With this characterization, we can compute the value of $B(n)$ by reading the bits of the binary form of the index $n$ from left
to right.  The initial 1 in the index tells us to start with an initial value of 1. Each time we encounter the bit 0 we double
our current value;  each time we encounter a 1 we halve the current value, unless the current value is 1, in which case the
value becomes and remains 0.  For example, suppose $n = 18 = 10010_{(2)}$.  The left-most bit is 1, and we start with an
initial value of 1; $B(1)=1$.  The next bit is 0, so we double our current value to 2; $B(10_{(2)})=2$.  The next bit is again 0,
so we again double our current value to 4;  $B(100_{(2)})=4$.  The next bit is 1, so we halve our current value back to 2;
$B(1001_{(2)})=2$.  The final bit is again 0, so our current value is doubled to our final value of 4;  $B(10010_{(2)})=4$.  On
the other hand, for $n=22=10110_{(2)}$, our current value starts at 1 and progresses to 2, back to 1, and then to 0, where it
remains; $B(1011\omega_{(2)})= 0$ for any bit string $\omega$.

The above observation provides us with a useful way of identifying the indices where $B(n)=0$.
\begin{lemma}
$B(n)=0$  if and only if, when reading the bits of the binary representation of $n$ from left to right, 
we at some
point have encountered more ones than zeros, not counting the initial 1.
\end{lemma}

We will call a word $\omega$ of length $2n$ over ${\mathcal A}$ a {\it Catalan} word if it consists of $n$ ones and $n$ zeros, and at
no point, reading from left to right, does it contain more ones than zeros.  The first few Catalan words are $\lambda$ (the null
word), 01, 0011, 0101, 000111, 001011, 001101, 010011, and 010101.  It is well known that the number of Catalan words of
length $2n$ is the $n$th Catalan number $\displaystyle{\binom{2n}{n}}/(n+1)\rule[-8pt]{0pt}{8pt}$ for all $n\ge 0$.  These Catalan numbers, which we denote 
$C(n)$, form sequence \href{https://oeis.org/A000108}{A000108} in the OEIS.

\begin{lemma}  There is a run of zeros in sequence $B$ 
beginning at index $n$ if any only if the binary representation of $n$ has the form
$1\omega1_{(2)}$ or $1\omega10_{(2)}$, where $\omega$ is a $($perhaps null$)$ Catalan word.
\end{lemma}

According to this lemma, the first few indices where runs of zeros begin are $11_{(2)} = 3$,\quad $110_{(2)} = 6$,\quad
$1011_{(2)}= 11$,\quad $10110_{(2)}= 22$,\quad $100111_{(2)}= 39$, and $1001110_{(2)}=78$.  This sequence is not currently in the OEIS.

\begin{proof}
Suppose $n = 1\omega1_{(2)}$, with $\omega$ a Catalan word.  
Then $n\!-\!1 = 1\omega0_{(2)}$.  By Lemma 4.6 we have $B(n)=0$ and $B(n\!-\!1) \neq 0$.   Likewise, if 
$n = 1\omega10_{(2)}$ then 
$n\!-\!1 = 1\omega01_{(2)}$, and again we have $B(n)=0$ and $B(n\!-\!1) \neq 0$.  In both cases, then,
a run of zeros starts at index $n$.

Now suppose $n$ is an index where $B(n)=0$ but $n$ is not of the above form.  There are three cases to consider.

\medskip
  1.  $n= 1\omega11_{(2)}$ where $\omega$ is a Catalan word.  Then $n\!-\!1 = 1\omega10_{(2)}$, so $B(n\!-\!1)=0$.

\smallskip 2.  $n=1\omega1\upsilon_{(2)}$ where $\omega$ is a Catalan word and $\upsilon$ is a word of length at least 2, not all zeros.
Then $n\!-\!1= 1\omega1\upsilon^{\prime}_{(2)}$, where $\upsilon^{\prime}_{(2)}= \upsilon_{(2)}-1$.  Again, $B(n\!-\!1) = 0$.

\smallskip3.  $n=1\omega1\upsilon_{(2)}$ where $\omega$ is a Catalan word and $\upsilon$ is a string of zeros of length at least 2.  Then
$n\!-\!1 = 1\omega0\upsilon^{\prime}_{(2)}$ where $\upsilon^{\prime}$ is a string of ones of length at least 2.  
Once again, $B(n\!-\!1) = 0$.
%\begin{enumerate}
%\item $n= 1w11_{(2)}$ where $w$ is a Catalan word.  Then $n\!-\!1 = 1w10_{(2)}$, so $B(n\!-\!1)=0$.
%\item $n=1w1x_{(2)}$ where $w$ is a Catalan word and $x$ is a word of length at least 2, not all zeros.
%Then $n\!-\!1= 1w1x^{\prime}_{(2)}$, where $x^{\prime}_{(2)}= x_{(2)}-1$.  Again, $B(n\!-\!1) = 0$.
%\item $n=1w1x_{(2)}$ where $w$ is a Catalan word and $x$ is a string of zeros of length at least 2.  Then
%$n\!-\!1 = 1w0x^{\prime}_{(2)}$ where $x^{\prime}$ is a string of 1s of length at least 2.  Once again, $B(n\!-\!1) = 0$.
%\end{enumerate}

\smallskip In all three cases, then, $B(n)= 0$ but is not the start of a run of zeros.
\end{proof}

We conclude that for $2^{2k-1} \le n <2^{2k}$, where the binary representation of $n$ contains $2k$ bits, there are exactly
$C(k-1)$ runs of zeros in sequence $B$, corresponding to the $C(k-1)$ Catalan sequences of length $2k-2$.
The longest run of zeros in this range is the last one, of length $A(k)$.  There are an
additional $C(k-1)$ runs when $2^{2k}\le n < 2^{2k+1}$ and $n$ is $2k+1$ bits long.  The longest run here is also the last one, 
of length $A(k+1)$.

The record-length runs of zeros in $B$ are thus runs number $C(0)=1$, $2C(0) = 2$, $2C(0)+C(1) = 3$, $2C(0)+2C(1) = 4$,
$2C(0)+2C(1)+ C(2) = 6$, $2C(0) + 2C(1) + 2C(2) = 8$,  $2C(0)+2C(1)+2C(2) +C(3)=13$, $2C(0)+2C(1)+2C(3)+2C(4) = 18$, and
so on.  These numbers form sequence \href{https://oeis.org/A155051}{A155051} in the OEIS, starting with $A155051(0)=1$.  This implies the second form of 
Zumkeller's conjecture.

\begin{theorem}
The $n$th record-length run of zeros in sequence $B$ has length $A(n)$ and is run number $A155051(n\!-\!1)$.  More concisely,
$A(n) = R(A155051(n\!-\!1))$.
\end{theorem}

%For Later Paper?
%
%Although Zumkeller does't mention it, it is true that the lengths of all runs of zeros in sequence $B$ are terms in sequence
%$A$, not just the record runs.
%
%\begin{theorem}
%For index $n$ satisfying $2^{2k-1}\le n <2^{2k}$, there is $C(0)=1$ run of zeros of length $A(2k-1)$ in sequence $B$, and for %$all $i$ satisfying 
%$1\le i \le k-1$ there are $C(i)-C(i-1)$ runs of length $A(2k-2i-1)$.  For index $n$ satisfying $2^{2k}\le n<2^{2k+1}$,
%there is $C(0)=1$ run of zeros of length $A{2k}$ in sequence $B$, and for $1\le i<k-1$ there are $C(i)-C(i-1)$ runs
%of length $A(2k-2i)$.
%\end{theorem}
%
%We illustrate this theorem for the case $k=5$.  For $512\le n<1024$ there is $C(0)=1$ run of length $A(9)= 341$,
%$C(1)-C(0) = 0$ runs of length $A(7) = 85$,  $C(2)-C(1)=1$ run of length $A(5) = 21$, $C(3)-C(2)= 3$ runs of length $A(3)=5$,
%and $C(4)-C(3)= 9$ runs of length $A(1)=1$.  For $1024\le n<2048$ there is $C(0)=1$ run of length $A(10)=682$,
%$C(1)-C(0) = 0$ runs of length $A(8) = 170$,  $C(2)-C(1)=1$ run of length $A(6) = 42$, $C(3)-C(2)= 3$ runs of length $A(4)=10$,
%and $C(4)-C(3)= 9$ runs of length $A(2)=2$.
%
%\begin{proof}
%
%\end{proof}

\section{The Question of N. J. A. Sloane}

Continuing in the OEIS listing for sequence $A$ (sequence A0000975) there is a link to an electronic paper 
\href{http://arxiv.org/abs/1310.4521}{Enveloping Operads and Bicoloured Noncrossing Configuration}                            \cite{CG} by F. Chapoton
and S. Giraudo.  In Table 2 of that paper the sequence 1, 2, 5, 10, 21, 42, 85 is listed twice.  In the OEIS Sloane asks
``Is the sequence in Table 2 this sequence [i.e., sequence $A$]?''.  In this section we demonstrate that the answer to 
this question is ``yes.''

  In the somewhat peculiar language of \cite{CG}, these numbers are the first few coefficients of the colored Hilbert
series giving the number of  bubbles of increasing arity, first based and then nonbased, of the 2-colored suboperad 
$\langle\langle{\color{blue} \mathbf \Delta}, {\mathbf \Delta}\rangle\rangle$ of the
2-colored operad {\sf Bubble} generated
by these two generators of arity 2.

In more simple languge, a {\it bubble} is a polygon consisting of a base edge and two or more nonbase edges, 
with each edge either colored blue or left uncolored.  The bubble is {\it based} if and only if the base edge is blue.  
The {\it border} of a bubble is
its set of nonbase edges. The {\it arity} of a bubble is the number of edges in its border.

The based bubbles in the 2-colored suboperad of interest start with a triangle of blue edges.  New bubbles are generated by either
replacing a blue border edge with two consecutive uncolored edges, or by replacing an uncolored border edge with two
consecutive blue edges.  Unbased bubbles in this suboperad are the same, but with blue and uncolored interchanged. 
See Figure \ref{bubble} for the based bubbles of arity 2, 3, and 4 in this suboperad.

\begin{figure}[h!]
\centerline{\includegraphics[width=4.93in,height=2.42in,keepaspectratio]{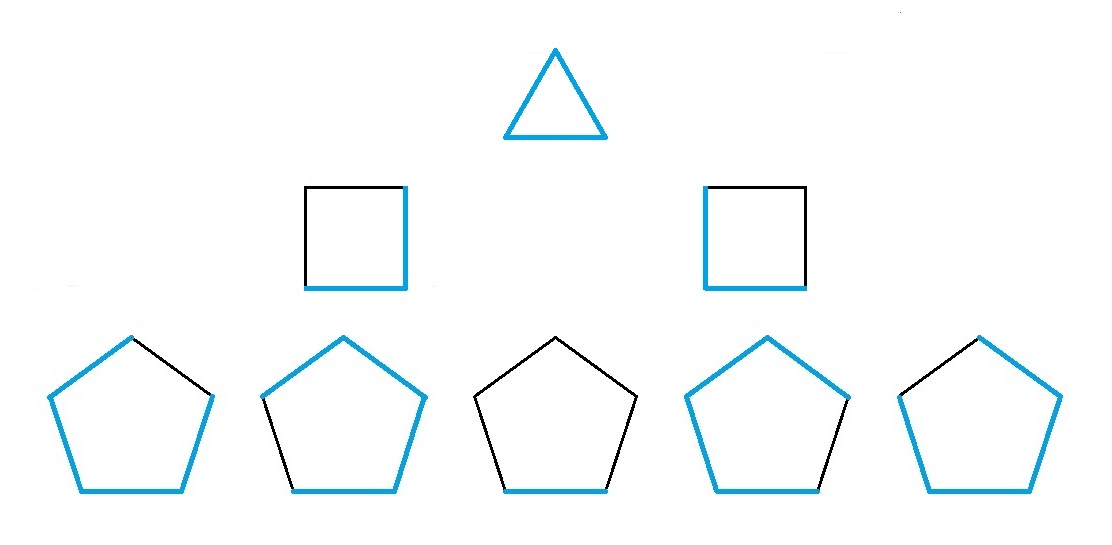}}
\caption{Based bubbles of arity 2, 3, and 4.}
\label{bubble}
\end{figure}

These based bubbles of arity $n$ are characterized in \cite{CG} as those with the following properties.
\begin{enumerate}
\item the number of blue edges in the border is congruent to $2n\!+\!1 \pmod{3}$ (and the number of uncolored edges in
the border is thus congruent to $2n\!-\!1 \pmod{3}$); and
\item  there exist 2 consecutive edges in the border that are either both blue or both uncolored.
\end{enumerate}
In addition,   the paper \cite{CG} gives 2-variable generating functions (or colored Hilbert series) for based and unbased bubbles in this suboperad.  Unfortunately these generating functions are flawed:  the numerator of the second term in the first generating function should be $z_1$,
 not $z_2$, and the numerator of the second term in the second generating function should be $z_2$, not $z_1$.  We now
produce our own count of these bubbles.

\begin{lemma}
Let $S(n)$ denote the number of strings of length $n$ over the alphablet $\{b,u\}$ such that the number of occurrances of the character $b$  is congruent to $2n\!+\!1\pmod{3}$.  Then $S(n) + S(n\!+\!1) = 2^n$.  Exactly one of these strings consists of alternating $b$ characters and $u$ characters.
\end{lemma}

\begin{proof}
Clearly the number of strings of length $n$ with $k$ occurrences of the character $b$ is ${\displaystyle \binom{n}{k}}$. So we have
\begin{align*}
S(n) + S(n\!+\!1) &= \sum_i\binom{n}{3i+2n+1} + \sum_i\binom{n+1}{3i+2(n+1) + 1}\\
              &=\sum_i\binom{n}{3i+2n+1} + \sum_i\left(\binom{n}{3i+2n+2}+\binom{n}{3i+2n+3}\right)\\
              &= \sum_i\binom{n}{i} = 2^n.
\end{align*}
The restrictions on the number of $b$ characters and $u$ characters
in a string imply that these counts cannot be equal, so an alternating string
must be of odd length.  If $n= 2k\!+\!1$ the number of $b$ characters must be congruent to 
$2(2k\!+\!1)+1 \equiv k \pmod{3}$ and the 
number
of $u$ characters must be congruent to $2(2k\!+\!1)-1 \equiv k+1\pmod{3}$, so there is a unique alternating string of length 
$n$, consisting of
 $k$ occurrences of the character $b$ and $k\!+\!1$ occurrences of the character $u$. 
\end{proof}

\begin{theorem}
The number of based $($or, equivalently, unbased$)$ bubbles of arity $n$ in the indicated suboperad of bubbles is $A(n\!-\!1)$.
\end{theorem}

\begin{proof}
Every based bubble in our suboperad can be constructed from a blue base and a border colored according to a string described
in Lemma 4.  There is one based bubble of arity 2, and the number of based bubbles of arity $n$ plus the number of based
bubbles of arity $n\!+\!1$ is $2^n-1$.  Thus the number of based bubbles satisfies Characterization 3 from Section 1.

The result for unbased bubbles follows by interchanging blue and uncolored in all the proofs.
\end{proof}

We conclude by illustrating an alternative approach to this counting problem: 
a one-to-one correspondence between the number of partitions into affinity groups described in
Section 2 and the based bubbles described here. To construct a bubble from an affinity group partition, we 
traverse the diagram of the affinity group partition counter-clockwise. If an A is followed by a B, a B by a C, or a C by an A, we
color the corresponding edge of the bubble blue. If an A is followed by a C, or a B by an A, or a C by a B, we leave the
corresponding edge of the bubble uncolored.  The correspondence between the 10 affinity group partitions for 6 people and the
10 bubbles of arity 5 is illustrated in Figure \ref{corr}.

\begin{figure}[h]
\centerline{\includegraphics[width=5.02in,height=3.85in,keepaspectratio]{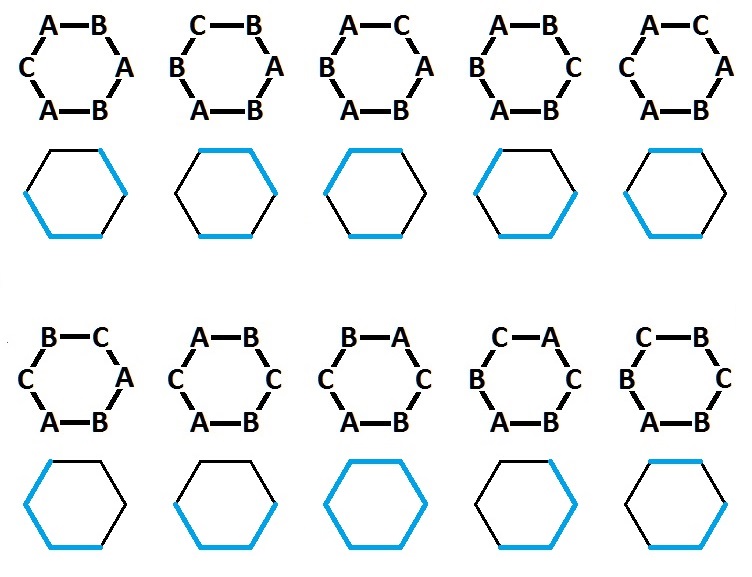}}
\caption{A one-to-one correspondence.}
\label{corr}
\end{figure}

The confirmation that this is indeed always a one-to-one correspondence is left to the reader.

\noindent\hrulefill

\footnotesize\noindent 2010  {\it Mathematics Subject Classification}: Primary 05A10, Secondary 11B37.

\footnotesize\noindent{\it Keywords}: Chinese Rings, baguenaudier, triangle numbers, palindromes, operads.

\begin{thebibliography}{9}
\bibitem{CG} Chapoton, Fr{\'e}d{\'e}ric, and Giraudo, Samuele, \href{http://arxiv.org/abs/1310.4521}{Enveloping Operads and Bicoloured Noncrossing Configuration}, arXiv preprint arXiv:1310.4521, 2013. Downloaded July 1, 2016.
\bibitem{Hinz} Hinz, A. M., Klav\v{z}ar, S., Milutinovi\'{c}, U., and Petr, C., {\sl The Tower of Hanoi---Myths and Maths.} Springer, Basel, 2013.
\bibitem{knuth}Knuth, Donald E., ``Problem 11151,'' {\sl American Mathematical Monthly} {\bf 112} (2005), 367.
\bibitem{Lossers} Lossers, O. P., ``Solution to Problem 11151: Partitions of a Circular Set,'' {\sl American Mathematical 
Monthly} {\bf 114} (2007), 265--266.
\bibitem{OEIS}
OEIS Foundation Inc., The On-Line Encyclopedia of Integer Sequences, \url{https://oeis.org}.  Downloads are current as of
July 1, 2016.
\end{thebibliography}
\end{document}